\documentclass[reqno]{amsart}
\usepackage[margin=1.5in]{geometry}
\usepackage{palatino}
\usepackage{mathpazo}
\usepackage{amsmath, amssymb, amsthm, mathrsfs}
\usepackage{tikz-cd}
\usepackage{hyperref}
\usepackage{enumerate}

\newtheorem{theorem}{Theorem}[section]
\newtheorem{introthm}{Theorem}

\newtheorem{lemma}[theorem]{Lemma}
\newtheorem{proposition}[theorem]{Proposition}

\newtheorem{corollary}[theorem]{Corollary}

\newtheorem*{theorem*}{Theorem}

\theoremstyle{definition}
\newtheorem{definition}[theorem]{Definition}

\newtheorem{example}[theorem]{Example}

\theoremstyle{remark}
\newtheorem{remark}[theorem]{Remark}

\newcommand{\LL}{\mathbb{L}}
\newcommand{\TT}{\mathbb{T}}
\newcommand{\Gm}{\mathbb{G}_m}
\newcommand{\cO}{\mathcal{O}}
\newcommand{\cL}{\mathcal{L}}
\newcommand{\cK}{\mathcal{K}}
\newcommand{\K}{\mathbb{K}}
\newcommand{\DR}{\mathbf{DR}}
\newcommand{\Spec}{\mathrm{Spec\,}}
\newcommand{\bfem}[1]{\textbf{\textit{#1}}}

\newcommand{\dCrit}{\mathrm{dCrit}}
\newcommand{\LFc}{\mathfrak{L}\mathcal{F}_c}
\newcommand{\Hom}{\mathrm{Hom}}

\makeatletter
\g@addto@macro{\endabstract}{\@setabstract}
\newcommand{\authorfootnotes}{\renewcommand\thefootnote{\@fnsymbol\c@footnote}}%
\makeatother

\begin{document}
\title[Equivariant Contact Darboux Quotients]{Equivariant Contact Darboux Quotients}
\author{Efe \.{I}zbudak}
\date{\today}

\begin{center}
  \LARGE
  Equivariant Contact Darboux Quotients \par \bigskip

  \normalsize
  \authorfootnotes
  Efe \.{I}zbudak\footnote{efe.izbudak@studium.uni-hamburg.de}\par \bigskip

  Department of Mathematics, Universit\"{a}t Hamburg\par \bigskip
\end{center}

{\renewcommand{\thefootnote}{}%
\footnotetext{2020 \emph{Mathematics Subject Classification.} Primary 14A30;
Secondary 53D10, 14N35.\par
\emph{Key words and phrases.} Derived algebraic geometry, shifted contact
structures, equivariant Darboux theorem, \'etale slices, derived critical
locus, Donaldson--Thomas invariants, quivers with potential.}}

\begin{abstract}
We prove an equivariant Darboux theorem for $-1$-shifted contact derived Artin stacks. Such a stack admits a smooth atlas by contact Darboux charts: the derived discriminant locus $\Delta\mathrm{loc}(s)$ of a relative section $s$, extended in degree $-2$ along the relative dimension of the chart. Near a point with linearly reductive stabilizer $G$ acting on the contact line through a character $\chi$, the stack is \'etale-locally the quotient $[\Delta\mathrm{loc}_G(s)/G]$, whose degree $-2$ generators are indexed by the Lie algebra of $G$ with differential a contact moment map. The cotangent complex sees these generators as $\mathrm{Lie}\,\ker(\chi)$ in degree $-2$; they vanish exactly when $\ker(\chi)$ is finite, and the quotient of the discriminant locus alone is a local model at no point where they are nonzero. For a quiver with potential the contact moment map is the moment map of the doubled quiver, the classical truncation is the representation space of the Jacobi algebra, and the degree $-2$ generators are carried exactly by the locus of positive-dimensional stabilizer.
\end{abstract}

\tableofcontents
\newpage
\section{Introduction and Summary}

The local structure of Artin stacks in derived symplectic geometry is given by
the Darboux theorem of Ben-Bassat, Brav, Bussi, and Joyce \cite[Theorem
2.10]{BenBassat2015}, which guarantees the existence of smooth atlases in Darboux form.
For the purposes of Donaldson-Thomas enumerative theory, this local model
refines, near a point with linearly reductive stabilizer $G$, to the derived
critical locus $\dCrit([f/G])$ of a $G$-invariant function $f$ on a quotient
stack $[U/G]$ \cite[Proposition 3.13]{KPS}; see also \cite{Descombes,
ParkPush}. On the atlas this is $\dCrit(f)$ extended by the Koszul--Tate
generators of the shifted moment map.

In the contact setting, Berktav develops the local theory of shifted
contact structures in \cite[Section 3]{kib1} and proves in \cite[Theorem
3.7]{kib2} that shifted contact derived Artin stacks admit smooth Darboux
atlases. Following the foundational results of \cite{kib1, kib2}, we have shown in \cite[Theorem
1.1]{IzbudakBerktav2026} that $\Gm$-equivariant quotients of symplectic derived Artin stacks are contact, and further built the contact-symplectic dictionary in \cite[Theorem 1.2]{IzbudakBerktav2026} and \cite[Theorem
1.1]{IzbudakBerktav2} by proving that the derived intersection of Legendrians is
contact, and constructing the global non-linear derived Legendrian categories
$\mathcal{F}_c(X)$ and $Leg_n$.

Note that, despite the existence of local model results for $-1$-shifted derived contact Artin stacks, equivariant \'etale vanishing cycles over a base field are computed on a quotient
by a reductive group, which a smooth atlas alone does not provide. This paper
supplies such quotients for $-1$-shifted contact stacks by the following theorem.

\begin{introthm}[Equivariant Contact Darboux Theorem] \label{thm:A}
Every quasi-separated, locally finitely presented $-1$-shifted contact derived
Artin stack $Z$ over $\K$ with affine stabilizers
\begin{enumerate}[(i)]
\item admits a smooth atlas by contact Darboux charts, and by the derived discriminant loci $\Delta\mathrm{loc}(s)$ themselves whenever the atlas can be taken \'etale (cf. Theorem \ref{thm:contact_darboux}(1)).
\item is, if $Z$ is a derived scheme, Zariski-locally around each
$\K$-point the derived discriminant locus of a regular function (cf. Corollary \ref{cor:schematic}).
\item is \'etale-locally equivalent to the quotient $[\Delta\mathrm{loc}_G(s)/G]$ for a
$\chi^{-1}$-semi-invariant function $s$ on a $G$-representation, near any point
$z \in Z(\K)$ with linearly reductive stabilizer $G$ acting on the contact line
through a character $\chi$ (cf. Theorem \ref{thm:contact_darboux}(2)).
\end{enumerate}
\end{introthm}

A contact Darboux chart is the Koszul--Tate algebra of $\Delta\mathrm{loc}(s)$
with generators adjoined in degree $-2$, and $\Delta\mathrm{loc}_G(s)$ is the
chart whose generators are indexed by the Lie algebra of $G$, with differential
the contact moment map (Definition \ref{def:equiv_discriminant}).

The paper's second main theorem ties the contact model to the symplectic theory, showing that symplectifying the contact model lands it exactly
on the local model of quotient Donaldson--Thomas theory.

\begin{introthm}[Symplectification of the local model] \label{thm:B}
Let $A_G$ be the Koszul--Tate algebra of $\Delta\mathrm{loc}_G(s)$ and
$\widetilde{A}_G := A_G[t, t^{-1}]$, with $\Gm$ scaling $t$ with weight $1$ and
$G$ acting on $t$ through character $\chi$. Then 
\begin{enumerate}[(i)]
\item $\widetilde{A}_G$ is the standard-form
Koszul--Tate cdga presenting $\dCrit([\,t \cdot s\,/\,G\,])$ (cf. Proposition \ref{prop:equiv_sympl}(1)).
\item $A_G = \widetilde{A}_G^{\Gm}$ (cf. Proposition \ref{prop:equiv_sympl}(1)).
\item The symplectification of $[\Delta\mathrm{loc}_G(s)/G]$ is
$\dCrit([\,t \cdot s\,/\,G\,])$ (cf. Proposition \ref{prop:equiv_sympl}(2)).
\end{enumerate}
\end{introthm}

The third main
theorem says that the degree $-2$ generators are forced, and identifies what
they measure.

\begin{introthm}[The obstruction] \label{thm:C}
Let $Z$ be a $-1$-shifted contact derived Artin stack and $z \in Z(\K)$ a point
whose stabilizer $G$ acts on the contact line through $\chi$. Then
\[ h^{-2}(\LL_Z|_z) \simeq \mathrm{Lie}\,\ker(\chi) \]
(cf. Proposition \ref{prop:degree_two}(1)).
Moreover $\Delta\mathrm{loc}_G(s) = \Delta\mathrm{loc}(s)$ if and only if
$\dim G = 0$ (cf. Lemma \ref{lem:equiv_KT_model}(2)), and consequently neither an \'etale morphism
$([\Delta\mathrm{loc}(s)/G], 0) \to (Z, z)$ inducing an isomorphism of
stabilizers nor a smooth atlas of $Z$ by contact Darboux schemes
$\Delta\mathrm{loc}(s)$ exists near a point with $\dim \ker(\chi) > 0$ (cf. Proposition \ref{prop:degree_two}(1)).
\end{introthm}

The quotient of the discriminant
locus alone is therefore not a local model in general, and the case in which it
fails at every point is the one of enumerative interest, as given in the following theorem.

\begin{introthm}[Quivers with potential] \label{thm:D}
Let $Q$ be a finite quiver, $W$ a potential, $d$ a dimension vector, and let
$G_d = \prod_i GL_{d_i}$ act on
$V = \mathrm{Rep}(Q,d)$ with $W_d = \mathrm{Tr}(W)$ and $\chi$ trivial. Then
\begin{enumerate}[(i)]
\item $[\Delta\mathrm{loc}_{G_d}(W_d)/G_d]$ is $-1$-shifted contact (cf. Proposition \ref{prop:quiver}(1)).
\item The contact
moment map is the moment map of the doubled quiver (cf. Proposition \ref{prop:quiver}(2)),
\[ \mu_i = \sum_{h(a) = i} x_a y_a - \sum_{t(a) = i} y_a x_a . \]
\item $h^{-2}(\LL_{[\Delta\mathrm{loc}_{G_d}(W_d)/G_d]}|_x) \simeq \mathrm{Lie}\, G_x$ at every point $x$ (cf. Corollary \ref{cor:quiver}(1)). 
\item The scalars act
trivially, so no point has finite stabilizer and Theorem \ref{thm:C} applies
everywhere (cf. Corollary \ref{cor:quiver}(1)).
\item If $W$ is homogeneous the classical truncation is
$\mathrm{Rep}(\mathrm{Jac}(Q,W), d)$ (cf. Corollary \ref{cor:quiver}(2)).
\item If every cycle of $W$ has length at least
$2$, the origin lies in the classical truncation and the Koszul--Tate
presentation of $\Delta\mathrm{loc}_{G_d}(W_d)$ is minimal there in degree $-2$ (cf. Corollary \ref{cor:quiver}(3)).
\end{enumerate}
\end{introthm}

Adding a framing makes the
action free on the stable locus and the obstruction disappears there, so the
degree $-2$ generators are carried exactly by the locus of positive-dimensional
stabilizer (Remark \ref{rem:quiver_framing}).

The companion paper \cite{Izbudak_Monodromic} constructs an $\ell$-adic monodromic
perverse sheaf on an oriented $-1$-shifted contact stack from the atlas of
Theorem \ref{thm:A} and studies the resulting Legendrian invariants; it uses only
the first assertion of Theorem \ref{thm:A}, and is logically independent of the
equivariant part proved here.

\medskip

\paragraph{\bfem{Conventions and notations}.}
Let $\K$ denote an algebraically closed field of characteristic zero. All commutative differential graded algebras (cdgas) are concentrated in nonpositive degrees and defined over $\K$ \cite{PTVV}. Classical $\K$-schemes are locally of finite type, and all derived $\K$-schemes and stacks $X$ are locally finitely presented.

\section{Recollection of Prior Work}
\label{sec:review}

In this section, we review the material used in the proofs, mainly, shifted symplectic and contact
structures with their local models, the derived discriminant locus, and the
\'etale slice theorem.

\subsection{Shifted Symplectic Structures and Lagrangian Intersections}
\label{subsec:symplectic}
All constructions take place in derived algebraic geometry over $\K$ in the sense of To\"en--Vezzosi \cite{ToenHAG}; see \cite{Toen2014} for a survey. Pantev, To\"en, Vaqui\'e, and Vezzosi \cite{PTVV} define shifted symplectic structures on derived Artin stacks as follows.
\vspace{1em}

Given a derived Artin stack $X$, denote by $\mathcal{A}^{p(, \mathrm{cl})}(X, n)$ the space of \bfem{(closed) $ p$-forms of degree $n$}, respectively. Let $L_{qcoh}(X)$ be the \bfem{stable $\infty$-category of quasi-coherent complexes} on a derived stack $X$. Then the deformation theory of $X$ is defined by the \bfem{cotangent complex} $\LL_{X} \in L_{qcoh}(X)$, which is the quasi-coherent complex of derived K\"{a}hler differentials. The \bfem{tangent complex} $\TT_{X}$ is defined as the derived dual complex $\TT_X \simeq \mathbb{R}\underline{\mathcal{H}om}(\LL_X, \cO_X)$.

\begin{definition}[Shifted Symplectic Structure \cite{PTVV}]\label{def:shifted_symplectic}
Let $X$ be a derived Artin stack. An \bfem{$n$-shifted symplectic structure} on $X$ is a closed 2-form $\omega \in \mathcal{A}^{2, \mathrm{cl}}(X, n)$ of degree $n$ such that the underlying 2-form is non-degenerate. Non-degeneracy requires that the induced morphism
\[ \Theta_{\omega} \colon \TT_{X} \longrightarrow \LL_{X}[n] \]
is an equivalence in $L_{\mathrm{qcoh}}(X)$.
\end{definition}

Lagrangian structures, derived Lagrangian intersections and the Darboux theorem
for $k$-shifted symplectic stacks with $k<0$ are as in \cite{PTVV} and
\cite[Theorem 2.10]{BenBassat2015}; we use them only through their contact counterparts
below.

\subsection{Shifted Contact Structures and Legendrian Categories}
\label{subsec:contact}
We now review shifted contact structures and Legendrian morphisms \cite{kib1, kib2, IzbudakBerktav2026, IzbudakBerktav2}.

\begin{definition}[Shifted Contact Structure \cite{kib1}]\label{def:contact_structure}
An \bfem{$n$-shifted contact structure} on a derived Artin stack $X$ consists of a choice of a line bundle $\cL$ and a non-degenerate Maurer-Cartan element that defines a contact form $\alpha \in \mathcal{A}^1(X, \cL, n)$. The contact distribution is defined as the homotopy fiber $\cK \simeq \mathrm{fib}(\TT_X \xrightarrow{\alpha^\vee} \cL[n])$ in $L_{\mathrm{qcoh}}(X)$ \cite[Definition 3.6]{kib1}.
\end{definition}

\begin{theorem}[Derived Symplectification \cite{kib1, kib2}]\label{thm:symplectification}
Let $(X, \cL, \alpha)$ be an $n$-shifted contact derived stack. The \bfem{derived symplectification} $\widetilde{X}$ is an $n$-shifted symplectic derived stack equipped with a principal $\Gm$-bundle projection $p \colon \widetilde{X} \to X$ and a free $\Gm$-action that scales its $n$-shifted symplectic form $\omega_{\widetilde{X}}$ by weight 1 \cite[Theorem 3.9]{kib2}, \cite[Definition 4.3, Theorem
4.7]{kib1}.
\end{theorem}

\begin{definition}[Legendrian Morphism \cite{IzbudakBerktav2026}]\label{def:legendrian}
A morphism $f \colon L \to X$ into an $n$-shifted contact stack $X$ is \bfem{Legendrian} if it is equipped with an isotropic structure $f^*\alpha \simeq 0$ inducing a stable fiber sequence $\cO_L \to \LL_{L/X} \otimes^{\mathbb{L}}_{\cO_L} f^*\cL[n-1] \to \TT_L$ \cite[Definition 2.39]{IzbudakBerktav2026}. This isotropic condition ensures that its derived symplectification $\widetilde{L} \to \widetilde{X}$ forms a $\Gm$-equivariant Lagrangian morphism \cite[proof of Theorem 4.1]{IzbudakBerktav2026}.
\end{definition}

Thus, the shifted contact structure descends to Legendrian intersections as follows.

\begin{theorem}[Derived Legendrian Intersections \cite{IzbudakBerktav2026, IzbudakBerktav2}]\label{thm:review_intersection}
Let $L_1, L_2 \to X$ be two Legendrian morphisms into an $n$-shifted contact stack $X$. The derived fiber product $L_{01} = L_1 \times_X^h L_2$ admits an $(n-1)$-shifted contact structure. Given three Legendrians $L_0, L_1, L_2$, the relative derived fiber product of two Legendrian morphisms $N_1 \to L_{01}$ and $N_2 \to L_{12}$ over $L_1$ inherits a Legendrian structure mapping into $L_{02}$.
\end{theorem}

Two facts about the symplectification $p \colon \widetilde{Z} \to Z$ of an
$n$-shifted contact stack $(Z, \cL, \alpha)$ are used throughout, and both are
formal consequences of its construction \cite[Definition 4.3, Theorem
4.7]{kib1}. First, the construction is functorial: automorphisms of the
tuple $(Z, \cL, \alpha)$ act on $\widetilde{Z}$, on the Liouville form $\lambda$
and on $\omega_{\widetilde{Z}} = d_{\DR}\lambda$, commuting with the fiberwise
$\Gm$-scaling; hence for $z \in Z(\K)$ the stabilizer $\mathrm{Aut}(z)$ acts on
the one-dimensional fiber $\cL_z$ through a character $\chi$, and on
$p^{-1}(z)$ through the same $\chi$. This $\chi$ is the character appearing in
all the statements below. Second, symplectification commutes with \'etale base
change. Principal bundles and the forms $\lambda$, $d_{\DR}\lambda$ pull back
along any morphism, and an \'etale $u \colon Y \to Z$ induces
$\TT_Y \simeq u^*\TT_Z$, so the non-degeneracy conditions pull back as well and
$\widetilde{Y} \simeq Y \times_Z \widetilde{Z}$ compatibly with $\lambda$ and
$\omega$.

\subsection{The 1-Jet Space and the Derived Discriminant Locus}
\label{subsec:discriminant}

The local models of Theorem \ref{thm:A} are built from derived discriminant
loci inside 1-jet spaces. We give their definition and Koszul--Tate
presentation, both used in Section \ref{sec:darboux}; the
extension needed over a quotient stack is Definition
\ref{def:equiv_discriminant}.

Let $V$ be a smooth affine $\K$-scheme with \'etale coordinates $x_1, \dots,
x_n$. The \bfem{1-jet space} $J^1(V) := T^*V \times \mathbb{A}^1$ carries the
coordinates $(x_i, y_i, z)$ and the canonical $0$-shifted contact form
\[ \alpha_{\mathrm{can}} = -d_{\DR}z + \sum_{i=1}^n y_i\, d_{\DR}x_i \]
with contact line bundle $\cO_{J^1(V)}$. This is the form $-d_{\DR}z + \lambda$
of \cite[Section 4.1]{kib2}, with $\lambda$ the tautological $1$-form of
$T^*V$, written in the coordinates $(x_i, y_i, z)$. Classically, it is the standard
contact form \cite[Example 3.3]{kib1}. 

The
\bfem{1-jet prolongation} of a regular function $s \in \cO(V)$ is the section
\[ j^1 s \colon V \longrightarrow J^1(V), \qquad
x \longmapsto (x,\, d s(x),\, s(x)), \]
and $z_V \colon V \to J^1(V)$, $x \mapsto (x, 0, 0)$, denotes the zero section.
Note that, both are Legendrian in the sense of Definition \ref{def:legendrian}, since
$(j^1 s)^* \alpha_{\mathrm{can}} = -d_{\DR}s + d_{\DR}s = 0$ and
$z_V^* \alpha_{\mathrm{can}} = 0$.

\begin{definition}[Derived discriminant locus]\label{def:discriminant}
The \bfem{derived discriminant locus} of $s \in \cO(V)$ is the derived
intersection of the 1-jet prolongation with the zero section,
\[ \Delta\mathrm{loc}(s) := V \times^h_{j^1 s,\; J^1(V),\; z_V} V . \]
By Theorem \ref{thm:review_intersection} it carries a canonical $-1$-shifted
contact structure.
\end{definition}

\begin{lemma}[Koszul--Tate model]\label{lem:KT_model}
Let $A$ be the quasi-free cdga generated over $\cO(V)$ by degree $-1$ generators
$y_1, \dots, y_n$ and $z$, with differential
\[ d(z) = s, \qquad d(y_i) = \frac{\partial s}{\partial x_i}, \]
and let $\alpha := d_{\DR}z + \sum_{i=1}^n y_i\, d_{\DR}x_i \in
\mathcal{A}^1(\Spec(A), \cO, -1)$. Then there is an equivalence
\[ \Spec(A) \;\simeq\; \Delta\mathrm{loc}(s) \]
of $-1$-shifted contact derived schemes under which $\alpha$ corresponds to the
contact form induced by $\alpha_{\mathrm{can}}$. The classical truncation of
$\Delta\mathrm{loc}(s)$ is the closed subscheme
$\{ s = 0 \} \cap \mathrm{Crit}(s) \subseteq V$.
\end{lemma}

\begin{proof}
The zero section $z_V$ is cut out in $J^1(V)$ by the regular sequence
$(y_1, \dots, y_n, z)$, so the derived intersection of Definition
\ref{def:discriminant} is presented by the derived tensor product
\[ \cO(V) \otimes^{\LL}_{\cO(J^1(V))} \cO(V) \]
in which one factor is resolved by the Koszul complex on that sequence. Pulling
the Koszul complex back along $j^1 s$ replaces its generators by the functions
\[ (j^1 s)^* y_i = \frac{\partial s}{\partial x_i}, \qquad (j^1 s)^* z = s, \]
so the resolution is the quasi-free cdga over $\cO(V)$ on degree $-1$ generators
$y_i$, $z$ with $d(y_i) = \partial s / \partial x_i$ and $d(z) = s$, which is
$A$. This identification is the identity on the generators $x_i, y_i, z$.

For the contact form, both sections are isotropic, so the induced structure on
the derived intersection is not a restriction of $\alpha_{\mathrm{can}}$. By the
proof of \cite[Theorem 4.1]{IzbudakBerktav2026}, the Legendrian null-homotopies
base-change to exact Lagrangian structures on the $\Gm$-saturated lifts of the
two sections in the symplectification of $J^1(V)$, the loop of these homotopies
transgresses to the $-1$-shifted symplectic structure on the derived
intersection of the lifts, and the quotient theorem
\cite[Theorem 1.1]{IzbudakBerktav2026} descends it. We compute this structure by
generating functions. The symplectification is $J^1(V) \times \Gm$ with
Liouville form $t\,\alpha_{\mathrm{can}}$. Setting \[u_i := t y_i, \quad \text{and} \quad \tau := -t\] identifies it with the locus $\tau \neq 0$ in
$T^*(V \times \mathbb{A}^1_z)$, carrying $t\,\alpha_{\mathrm{can}} =
\sum_i u_i\, d_{\DR}x_i + \tau\, d_{\DR}z$ to the tautological form. The lift
of $j^1 s$ is the punctured conormal of $\{z = s(x)\}$, the lift of $z_V$ that
of $\{z = 0\}$, and the base-changed homotopies are the strict vanishing of the
tautological form on each conormal. 

In the polarization with base
$(x_i, \tau)$ and momenta $(u_i, -z)$ the tautological form changes by the
exact term $-d_{\DR}(\tau z)$, so the first lift becomes the graph of
$d_{\DR}(t \cdot s)$ with its canonical primitive $t \cdot s = -\tau s$, and
the second the zero section with primitive $0$: the intersection is
$\dCrit(t \cdot s)$ over $V \times \Gm$ with its standard Koszul--Tate form
\cite[Example 5.15]{Brav}. 

Concretely, with \[\widetilde{y}_i := t y_i \quad \text{and} \quad \zeta := z \,\text{in}\, A[t, t^{-1}],\] 
the generators satisfy \[d(\widetilde{y}_i) =
\partial(ts)/\partial x_i\quad \text{and}\quad d(\zeta) = \partial(ts)/\partial t,\] with
\[\omega = d_{\DR}t\, d_{\DR}\zeta + \sum_i d_{\DR}\widetilde{y}_i\,
d_{\DR}x_i,\] all of it $\Gm$-equivariantly for the weight-one scaling of $t$.

Contracting $\omega$ with the fundamental vector field, as in the construction
of \cite[Theorem 1.1]{IzbudakBerktav2026}, gives
\[\iota_{t\,\partial/\partial t}\,\omega = t\,\bigl(d_{\DR}z + \sum_i y_i\,
d_{\DR}x_i\bigr),\] of weight one and its reduction under the trivialization by
$t$ is the displayed \[\alpha = d_{\DR}z + \sum_i y_i\, d_{\DR}x_i,\] the contact
Darboux normal form of \cite[Section 3]{kib1} (where the distinguished
generator is $-z$), \cite[Theorem 3.7]{kib2} with Hamiltonian $s$. Proposition \ref{prop:equiv_sympl} below extends this
computation equivariantly. Taking $H^0$ of $A$ imposes the
relations $\partial s/\partial x_i = 0$ and $s = 0$, giving the stated classical
truncation.
\end{proof}

The derived discriminant locus differs from the derived critical locus by the
additional equation $s = 0$, resolved by the generator $z$. Equivalently,
$\Delta\mathrm{loc}(s)$ is the derived zero locus of the $1$-jet of $s$, whereas
$\dCrit(s)$ is the derived zero locus of $d s$. The generator $z$ carries the
contact direction, and Step 3 of the proof of Theorem \ref{thm:contact_darboux}
identifies the symplectification of $\Delta\mathrm{loc}(s)$ with
$\dCrit(t \cdot s)$. Over a quotient stack, both sides acquire further
Koszul--Tate generators in degree $-2$ and the resulting model is Definition
\ref{def:equiv_discriminant}.

\subsection{Quivers with Potential}
\label{subsec:quiverprelim}

The application in Section \ref{subsec:quivers} uses the following standard
notions, for which we follow \cite{Ginzburg}.

\begin{definition}[Quiver and path algebra]\label{def:quiver}\mbox{} \medskip
\begin{enumerate}[(i)]
\item A \bfem{quiver} is a datum $Q = (Q_0, Q_1, t, h)$ consisting of two finite sets,
the \bfem{vertices} $Q_0$ and the \bfem{arrows} $Q_1$, together with two maps
$t, h \colon Q_1 \to Q_0$. 

\item We write $a \colon t(a) \to h(a)$ and call $t(a)$ and
$h(a)$ the \bfem{tail} and the \bfem{head} of $a$. 

\item A \bfem{path} of length
$\ell \geq 1$ is a word $a_\ell \cdots a_1$ in the arrows with
$h(a_j) = t(a_{j+1})$ for $1 \leq j < \ell$, and for each vertex $i$ there is a
path $e_i$ of length $0$. 
\item A path is a \bfem{cycle} if $h(a_\ell) = t(a_1)$. \item The \bfem{path algebra} $\K Q$ is the $\K$-vector space with basis the paths, with
multiplication given by concatenation where composable and by $0$ otherwise. It
is associative with unit $\sum_{i \in Q_0} e_i$.
\end{enumerate}
\end{definition}

\begin{definition}[Representations]\label{def:quiver_rep}
A \bfem{representation} of $Q$ of \bfem{dimension vector}
$d \in \mathbb{N}^{Q_0}$ is a choice of linear map $x_a \colon \K^{d_{t(a)}} \to
\K^{d_{h(a)}}$ for every arrow $a$, equivalently a $\K Q$-module structure on
$\bigoplus_i \K^{d_i}$ in which $e_i$ acts as the projection onto $\K^{d_i}$.
These form the affine space
\[ \mathrm{Rep}(Q, d) := \bigoplus_{a \in Q_1}
\Hom\bigl(\K^{d_{t(a)}}, \K^{d_{h(a)}}\bigr) , \]
on which $G_d := \prod_{i \in Q_0} GL_{d_i}$ acts by
$(g \cdot x)_a = g_{h(a)}\, x_a\, g_{t(a)}^{-1}$, the orbits being the
isomorphism classes of $d$-dimensional representations. 

For an algebra
$\K Q / I$ we write $\mathrm{Rep}(\K Q/I, d) \subseteq \mathrm{Rep}(Q,d)$ for
the closed subscheme where the relations in $I$ vanish.
\end{definition}

\begin{definition}[Potential, cyclic derivative, Jacobi algebra]\mbox{}\medskip
\label{def:potential}
\begin{enumerate}[(i)]
\item A \bfem{potential} is an element $W \in \K Q / [\K Q, \K Q]$. The quotient has
as a basis the cycles of $Q$ taken up to cyclic rotation, so $W$ is a finite
$\K$-linear combination of such. 
\item For $a \in Q_1$ the \bfem{cyclic derivative}
$\partial_a \colon \K Q / [\K Q, \K Q] \to \K Q$ is the linear map given on a
cycle $c = a_\ell \cdots a_1$ by
\[ \partial_a c := \sum_{j \,:\, a_j = a}
a_{j-1} \cdots a_1\, a_\ell \cdots a_{j+1} , \]
a path from $h(a)$ to $t(a)$. 
\item The \bfem{Jacobi algebra} of $(Q, W)$ is
\[ \mathrm{Jac}(Q, W) := \K Q / \bigl( \partial_a W : a \in Q_1 \bigr) . \]
\end{enumerate}
\end{definition}

A cyclic path evaluates on a representation $x$ to an endomorphism of the vector
space at the vertex where it begins and ends, and its trace does not change under
cyclic rotation, so a potential induces a regular function
\[ W_d := \mathrm{Tr}(W) \in \cO(\mathrm{Rep}(Q,d))^{G_d} , \]
invariant because trace is a conjugation invariant. Differentiating in the
entries of $x_a$ against the trace pairing gives $\partial W_d / \partial x_a =
\partial_a W$ evaluated on $x$, so
\begin{equation}\label{eq:crit_jac}
\mathrm{Crit}(W_d) = \mathrm{Rep}\bigl(\mathrm{Jac}(Q,W), d\bigr)
\end{equation}
\cite[Section 2.3]{Ginzburg}. This is the local model of quiver
Donaldson--Thomas theory \cite{Szendroi}.

\begin{definition}[Doubled quiver]\label{def:doubled}
The \bfem{doubled quiver} $\overline{Q}$ has $\overline{Q}_0 = Q_0$ and one
arrow $a^* \colon h(a) \to t(a)$ for each $a \in Q_1$ in addition to the arrows
of $Q$.
\end{definition}

One has $\mathrm{Rep}(\overline{Q}, d) = T^*\mathrm{Rep}(Q, d)$ as a symplectic
$G_d$-representation, with moment map
\begin{equation}\label{eq:preproj}
\mu_i(x, y) = \sum_{h(a) = i} x_a y_{a^*} \; - \sum_{t(a) = i} y_{a^*} x_a ,
\qquad i \in Q_0 ,
\end{equation}
whose zero locus cuts out the representations of the preprojective algebra
\cite{CB}. The expression \eqref{eq:preproj} is the one that will appear in
degree $-2$ in Section \ref{subsec:quivers}, with the cotangent variables
$y_{a^*}$ replaced by Koszul--Tate generators of degree $-1$.

\subsection{\'Etale Slices and Equivariant Darboux Models}\label{subsec:slices}

Recall that an affine algebraic group $G$ over $\K$ is \bfem{linearly reductive} if the
category of its rational representations is semisimple \cite[Section 1.5]{AHR}. In characteristic zero
this coincides with reductivity. Semisimplicity makes the
$\chi$-isotypic projector $R_\chi$ on rational $G$-modules exact and
functorial for every character $\chi$ of $G$, and we write $R :=
R_{\mathrm{triv}}$ for the Reynolds operator. The consequences used in Section
\ref{sec:darboux} are the following.

\begin{lemma}[Equivariant lifting principle]\label{lem:reynolds}
Let $G$ be linearly reductive over $\K$, $\mathrm{char}\, \K = 0$, and let
$\chi$ be a character of $G$.
\begin{enumerate}
\item Every surjection $M \twoheadrightarrow N$ of rational $G$-modules admits a
$G$-equivariant section, and every $\chi$-semi-invariant element of $N$ lifts to
a $\chi$-semi-invariant element of $M$.
\item Let $(B, d)$ be a cdga with $G$-action and let $h$ be a contracting
homotopy for a $G$-stable acyclic subcomplex, or a homotopy witnessing the
exactness of a $\chi$-semi-invariant closed form. Then $R(h)$, respectively
$R_\chi(h)$, is a $G$-equivariant homotopy witnessing the same identity.
\end{enumerate}
\end{lemma}

\begin{proof}~
\begin{enumerate}
\item By semisimplicity, the kernel admits a $G$-stable complement. Applying the
exact projector $R_\chi$ to any set-theoretic lift of a semi-invariant element
produces a semi-invariant lift, since $G$-equivariant maps commute with
$R_\chi$.
\item The defining identity of a homotopy, $dh + hd = \iota$ or $h \mapsto dh =
\beta$, is linear in $h$ and $G$-equivariant in the other terms (for the second,
$\chi$-semi-invariant). Applying the exact projector to the identity replaces
$h$ by $R(h)$ or $R_\chi(h)$ without changing the right-hand side.
\end{enumerate}
\end{proof}

Averaging also trivializes an equivariant line bundle near a fixed point, at
the cost of remembering a character. Let $G$ act on an affine derived
$\K$-scheme $S$ of finite presentation fixing a classical point $s_0$, and let
$\cL$ be a $G$-equivariant line bundle on $S$ on whose fiber $\cL_{s_0}$ the
group acts through $\chi$. The global sections $\Gamma(S, \cL)$ form a rational
$G$-module and evaluation at $s_0$ is a surjection of $G$-modules, so exactness
of the $\chi$-isotypic projector produces a $\chi$-semi-invariant section $e$
with $e(s_0) \neq 0$ whose non-vanishing locus is a $G$-invariant affine open
$S' \ni s_0$ on which $e$ trivializes $\cL$, giving
$\cL|_{S'} \simeq \cO_{S'} \otimes \chi$ equivariantly, with associated
principal $\Gm$-bundle $S' \times \Gm$ and $G$-action
$g \cdot (x, t) = (g \cdot x, \chi(g) t)$. We refer to this as the
\bfem{semi-invariant trivialization} and use it without further comment.

The same averaging makes the minimal standard form presentations of
\cite[Section 4]{Brav} equivariant.
\vspace{1em}

\begin{proposition}[Equivariant minimal standard form]
\label{prop:equiv_standard_form}
Let $G$ be linearly reductive over $\K$, let $B$ be a cdga of finite
presentation in nonpositive degrees with $G$-action, and let $s_0$ be a
$G$-fixed classical point of $\Spec(B)$. Then $B$ admits, \'etale locally at
$s_0$, a $G$-equivariant minimal standard form presentation. 

That is, $B$ admits a quasi-free cdga
$A$ whose generators in each degree form finite-dimensional
$G$-representations, whose degree-$0$ generators realize the $G$-module
$\mathfrak{m}_{s_0}/\mathfrak{m}_{s_0}^2$, together with an equivariant
differential and an equivariant equivalence $A \xrightarrow{\sim} B$.
\end{proposition}

\begin{proof}
We build $A$ by the standard induction on degree. In degree $0$, choose a
$G$-stable lift of $\mathfrak{m}_{s_0}/\mathfrak{m}_{s_0}^2$ to
generators, which is possible by Lemma \ref{lem:reynolds}(1). Assume
$A^{\geq -k} \to B$ has been constructed equivariantly and is
$(k{+}1)$-connective. The obstruction module
\[M := H^{-k-1}(\mathrm{fib}(A^{\geq -k} \to B))\]
is a rational $G$-module,
finitely generated over $H^0(B)$ since $B$ is of finite presentation.
Evaluation at $s_0$ is a $G$-equivariant surjection of $M$ onto its fiber at
$s_0$, a finite-dimensional $G$-representation, and Lemma
\ref{lem:reynolds}(1) provides a finite-dimensional $G$-stable subspace
$W \subseteq M$ mapping isomorphically onto this fiber. 

After localizing at
$s_0$, $W$ generates $M$ by Nakayama, and this is the minimal choice of
generators in this degree. Adjoin $W$ in degree $-k-1$ and extend the
differential by an equivariant lift, using Lemma \ref{lem:reynolds}(1), of the
map classifying the obstruction.

Minimality at $s_0$ is the vanishing of the induced differential on generators
at $s_0$. This is preserved by the averaging, since $s_0$ is $G$-fixed and
$R_\chi$ acts on values at $s_0$.
\end{proof}

Finally, the slices to which all of this will be applied come from the \'etale
slice theorem.

\begin{theorem}[\'Etale slice theorem \cite{AHR}]\label{thm:AHR}
Let $X$ be a quasi-separated algebraic stack, locally of finite type over $\K$,
with affine stabilizers, and let $x \in X(\K)$ be a point whose stabilizer $G$
is linearly reductive. Then there exist an affine $\K$-scheme $\Spec(A)$ with a
$G$-action fixing a point $w$ and a pointed \'etale morphism
\[ ([\Spec(A)/G],\, w) \longrightarrow (X, x) \]
inducing an isomorphism of stabilizers at $w$ \cite[Theorem 1.1]{AHR}.
\end{theorem}

In the symplectic setting, equivariant refinements of the Darboux theorems of
\cite{Brav, BenBassat2015} are established by Descombes for torus actions preserving the
symplectic form \cite{Descombes}, by Park in the relative setting
\cite{ParkPush}, and by Kinjo--Park--Safronov for stacks with reductive
stabilizers \cite[Proposition 3.13]{KPS}. Graded versions, with the potential
homogeneous of weight one, arise in dimensional reduction \cite{Kinjo}.
Proposition \ref{prop:equiv_darboux} and Theorem \ref{thm:contact_darboux} are
the contact counterparts.

\section{The Equivariant Stacky Contact Darboux Theorem}
\label{sec:darboux}

This section proves the four theorems of the introduction using the tools of
\S\ref{subsec:contact} and \S\ref{subsec:slices}. We define the local model
$\Delta\mathrm{loc}_G(s)$ and establish its two essential properties, which are
Theorems \ref{thm:B} and \ref{thm:C}: its symplectification is the derived
critical locus of $t \cdot s$ on a quotient stack, and its degree $-2$
cohomology is $\mathrm{Lie}\,\ker(\chi)$, which is what forces the generators
$w_\xi$ (\S\ref{subsec:eqdisc}). The affine case of Theorem \ref{thm:A} is
proved in \S\ref{subsec:affine} and globalized by the slice theorem in
\S\ref{subsec:global}; Theorem \ref{thm:D} occupies \S\ref{subsec:quivers}.

\subsection{The Equivariant Discriminant Locus}
\label{subsec:eqdisc}

Over a quotient stack the derived discriminant locus of Definition
\ref{def:discriminant} acquires Koszul--Tate generators in degree $-2$, dual to
the copy of $\mathfrak{g}^\vee$ that the quotient contributes to the cotangent
complex in degree $1$. We give the resulting local model, and in Proposition
\ref{prop:degree_two} the cohomological invariant that forces it.

\begin{definition}[Contact Darboux charts, the equivariant discriminant
locus]\label{def:equiv_discriminant}\mbox{}\medskip
\begin{enumerate}[(i)]
\item A \bfem{contact Darboux chart} is $\Spec(A)$ for $A$ the Koszul--Tate algebra of
Lemma \ref{lem:KT_model} for some $s \in \cO(V)$, extended by finitely many
generators in degree $-2$, together with the form $\alpha = d_{\DR}z + \sum_i
y_i\, d_{\DR}x_i$. A chart with no generators in degree $-2$ is
$\Delta\mathrm{loc}(s)$.

\item Let $G$ be linearly reductive with Lie algebra $\mathfrak{g}$, let $\chi
\colon G \to \Gm$ be a character, let $V$ be a finite-dimensional
$G$-representation with coordinates $x_1, \dots, x_n$, and let $s \in \cO(V)$ be
$\chi^{-1}$-semi-invariant, that is, $s(g \cdot x) = \chi(g)^{-1} s(x)$. Write
$a_\xi = \sum_i a_\xi(x_i)\, \partial/\partial x_i$ for the vector field
generating the action of $\xi \in \mathfrak{g}$, where
\[ a_\xi(h)(x) := \frac{d}{d\epsilon}\Big|_{\epsilon = 0}
h\bigl(\exp(\epsilon \xi) \cdot x\bigr), \qquad h \in \cO(V), \]
so that $a_\xi(s) = -\langle d\chi, \xi \rangle s$. The \bfem{equivariant
discriminant locus} $\Delta\mathrm{loc}_G(s) := \Spec(A_G)$ of the datum
$(G, \chi, V, s)$ is the contact Darboux chart whose degree $-2$ generators are
$w_\xi$, linear in $\xi \in \mathfrak{g}$, with differential the \bfem{contact
moment map}
\[ \langle \mu, \xi \rangle := d(w_\xi)
= \sum_{i=1}^n a_\xi(x_i)\, y_i + \langle d\chi, \xi \rangle\, z , \]
the $G$-action on the generators being the one for which $\alpha$ is
$\chi^{-1}$-semi-invariant. Here and below, the $G$-module spanned by a set of
generators is labelled by the pullback action: a generator $u$ with
$g^* u = \chi(g)^{-1} u$ is said to span $\chi^{-1}$, matching the convention
$s(g \cdot x) = \chi(g)^{-1} s(x)$ for semi-invariants. Under the left-module
convention $g \cdot u := (g^{-1})^* u$ every character twist below is inverted.
\end{enumerate}
\end{definition}

\begin{lemma}[Equivariant Koszul--Tate model]\label{lem:equiv_KT_model}
With the notation of Definition \ref{def:equiv_discriminant}:
\begin{enumerate}
\item $d^2 = 0$, and $G$ acts on the generators of $A_G$ through the modules
$V^\vee$ in degree $0$, $V \otimes \chi^{-1}$ (the $y_i$) and $\chi^{-1}$ (the
generator $z$) in degree $-1$, and $\mathfrak{g} \otimes \chi^{-1}$ (the $w_\xi$)
in degree $-2$, where $\mathfrak{g}$ carries the adjoint action.
\item $H^0(A_G) = \cO(\{s = 0\} \cap \mathrm{Crit}(s))$, so
$\Delta\mathrm{loc}_G(s)$ and $\Delta\mathrm{loc}(s)$ have the same classical
truncation, and $\Delta\mathrm{loc}_G(s) = \Delta\mathrm{loc}(s)$ if and only if
$\dim G = 0$.
\end{enumerate}
\end{lemma}

\begin{proof}~
\begin{enumerate}
\item Only $d^2(w_\xi)$ requires checking:
\[ d^2(w_\xi) = \sum_{i=1}^n a_\xi(x_i)\, \frac{\partial s}{\partial x_i}
+ \langle d\chi, \xi \rangle\, s = a_\xi(s) + \langle d\chi, \xi \rangle s = 0 \]
by the displayed consequence of semi-invariance. For the weights, the degree $0$
generators span $V^\vee$ by construction, and the relation $g^* \alpha =
\chi(g)^{-1}\alpha$ holds if and only if the $y_i$ span $V \otimes \chi^{-1}$ and
$z$ spans $\chi^{-1}$. For these weights, $d(z) = s$ and $d(y_i) = \partial
s/\partial x_i$ are equivariant precisely because $s$ is $\chi^{-1}$-semi-invariant.

In degree $-2$, the first summand of $d(w_\xi)$ is the
contraction of the vector field $a_\xi$ with $\alpha - d_{\DR}z$, hence
transforms by $\chi^{-1}$ and by the adjoint action in $\xi$, while $d\chi$ is
$\mathrm{Ad}$-invariant, so the second summand transforms the same way. 

Thus
$g^* w_\xi = \chi(g)^{-1} w_{\mathrm{Ad}_{g^{-1}} \xi}$, which is the module
$\mathfrak{g} \otimes \chi^{-1}$.

\item The generators $w_\xi$ lie in degree $-2$ and therefore do not contribute
to $H^0$, which is $\cO(V)$ modulo the image of $d$ on degree $-1$ elements, $\cO(V)/(\partial s/\partial x_1, \dots, \partial s/\partial x_n, s)$.
Since
$\mathrm{char}\, \K = 0$, we have $\dim \mathfrak{g} = \dim G$, so the two cdgas
agree exactly when $\dim G = 0$.
\end{enumerate}
\vspace{-1.5em}
\end{proof}

Inverting $t$ symplectifies the chart, and the result is a derived critical
locus in the standard form of \cite{Brav, KPS}. This is what ties the model to
the symplectic theory.

\begin{proposition}[Symplectification of the equivariant chart]
\label{prop:equiv_sympl}
With the notation of Definition \ref{def:equiv_discriminant}:
\begin{enumerate}
\item Set $\widetilde{A}_G := A_G[t, t^{-1}]$, with $\Gm$ scaling $t$ with weight
$1$ and $G$ acting on $t$ through $\chi$, and
\[ \widetilde{y}_i := t y_i, \qquad \zeta := z, \qquad
\widetilde{w}_\xi := t w_\xi . \]
Then $\widetilde{A}_G$ is the standard-form Koszul--Tate cdga presenting
$\dCrit([\,\widetilde{f}\,/\,G\,])$, the derived critical locus of the
$G$-invariant potential $\widetilde{f}(x,t) = t \cdot s(x)$ on the quotient stack
$[(V \times \Gm)/G]$, and $A_G = \widetilde{A}_G^{\Gm}$.
\item The Liouville form $\lambda := t\alpha$ has $\Gm$-weight $1$ and satisfies
$d_{\DR}\lambda = \omega$, the standard $-1$-shifted symplectic form of
$\dCrit([\,\widetilde{f}\,/\,G\,])$. Consequently $[\Delta\mathrm{loc}_G(s)/G]$
carries a $-1$-shifted contact structure with contact line $\cO \otimes \chi$
and contact form $\alpha$, whose symplectification is
$\dCrit([\,\widetilde{f}\,/\,G\,])$. The four spaces form a Cartesian square
\[
\begin{tikzcd}[row sep=large, column sep=large]
\Spec(\widetilde{A}_G) \arrow[r] \arrow[d, "\Gm"']
\arrow[dr, phantom, "\lrcorner", very near start] &
\dCrit([\,\widetilde{f}\,/\,G\,]) \arrow[d, "\Gm"] \\
\Delta\mathrm{loc}_G(s) \arrow[r] & {[\Delta\mathrm{loc}_G(s)/G]}
\end{tikzcd}
\]
in which the horizontal maps are the $G$-torsors presenting the two quotient
stacks and the vertical maps are the $\Gm$-torsors of the symplectification.
\end{enumerate}
\end{proposition}

\begin{proof}~
\begin{enumerate}
\item In degree $0$ one has $\cO(V)[t, t^{-1}] = \cO(V \times \Gm)$. In degree
$-1$,
\[ d(\widetilde{y}_i) = t\, \frac{\partial s}{\partial x_i}
= \frac{\partial \widetilde{f}}{\partial x_i}, \qquad
d(\zeta) = s = \frac{\partial \widetilde{f}}{\partial t}, \]
so $\widetilde{y}_i$ and $\zeta$ are the Koszul--Tate generators dual to
$\Omega^1_{V \times \Gm}$ \cite[Theorem 5.18]{Brav}. In degree $-2$,
\[ d(\widetilde{w}_\xi) = t\, d(w_\xi)
= \sum_{i=1}^n a_\xi(x_i)\, \widetilde{y}_i
+ a_\xi(t)\, \zeta
= \sum_{i=1}^n a_\xi(x_i)\, \widetilde{y}_i
+ \langle d\chi, \xi \rangle\, t \zeta, \]
since $G$ acts on $t$ through $\chi$ and hence $a_\xi(t) = \langle d\chi, \xi
\rangle t$. This is the shifted moment map of the $G$-action: under the
identification of the degree $-1$ generators of $\widetilde{A}_G$ with
$T_{V \times \Gm}$, it is the image of the vector field $a_\xi$, and the
$\widetilde{w}_\xi$ are the generators dual to the copy of $\mathfrak{g}^\vee$
sitting in degree $1$ of $\LL_{[(V \times \Gm)/G]}$. It is a cycle by Lemma
\ref{lem:equiv_KT_model}(1), since $\widetilde{f}$ is $G$-invariant, $t$ and $s$
carrying the characters $\chi$ and $\chi^{-1}$. This is the standard form of the
derived critical locus of an invariant function on a quotient stack
\cite[Proposition 3.13]{KPS}, \cite[Theorem 2.10]{BenBassat2015}. Finally the generators
$x_i, y_i, z, w_\xi$ all have $\Gm$-weight $0$ and $t$ has weight $1$, so
$\widetilde{A}_G^{\Gm} = A_G$.

\item The standard $-1$-shifted symplectic form of
$\dCrit([\,\widetilde{f}\,/\,G\,])$ is, on the atlas,
\[ \omega = d_{\DR}t \wedge d_{\DR}\zeta
+ \sum_{i=1}^n d_{\DR}\widetilde{y}_i \wedge d_{\DR}x_i \]
by \cite[Theorem 5.18]{Brav}. Substituting $\zeta = z$ and $\widetilde{y}_i = t y_i$
gives
\[ \omega = d_{\DR}t \wedge d_{\DR}z
+ \sum_{i=1}^n d_{\DR}(t y_i) \wedge d_{\DR}x_i, \]
and the Leibniz rule gives
\[ \omega = d_{\DR}t \wedge d_{\DR}z
+ \sum_{i=1}^n (d_{\DR}t \cdot y_i + t \cdot d_{\DR}y_i) \wedge d_{\DR}x_i. \]
Rearranging,
\begin{align*}
\omega &= d_{\DR}t \wedge \biggl(d_{\DR}z + \sum_{i=1}^n y_i\, d_{\DR}x_i\biggr)
+ t \sum_{i=1}^n d_{\DR}y_i \wedge d_{\DR}x_i \\
&= d_{\DR}\biggl(t \biggl(d_{\DR}z + \sum_{i=1}^n y_i\, d_{\DR}x_i\biggr)\biggr)
= d_{\DR}(t\alpha),
\end{align*}
so the extracted weight $0$ form $\alpha$ is the standard Darboux contact
$1$-form in local coordinates \cite{kib1, kib2}. By (1) the $\Gm$-action on
$\dCrit([\,\widetilde{f}\,/\,G\,])$ is free with quotient
$[\Delta\mathrm{loc}_G(s)/G]$, and the associated line bundle is $\cO \otimes
\chi$ because $G$ acts on $t$ through $\chi$. By the quotient theorem for
weight-one $\Gm$-actions \cite[Theorem 1.1]{IzbudakBerktav2026}, this datum
descends to a $-1$-shifted contact structure with contact form $\alpha$: the
descended form is the contraction of $\omega$ with the fundamental vector field
of the $\Gm$-action. Here,
$\iota_{t\,\partial/\partial t}\, d_{\DR}(t\alpha) = t\alpha$, whose weight-$0$
reduction under the trivialization by $t$ is $\alpha$. The resulting torsor,
with $\lambda = t\alpha$ and $\omega = d_{\DR}\lambda$, is the symplectification
in the sense of \cite[Definition 4.3]{kib1}, \cite[Theorem 3.9]{kib2}.
\end{enumerate}
\end{proof}

Two independent group actions are thus in play, doing different work: the
vertical $\Gm$ direction produces the symplectic structure, and the horizontal
$G$ direction produces the degree $-2$ generators.

\begin{remark}[Generators, degrees and weights]\label{rem:generator_table}
The bookkeeping of Lemma \ref{lem:equiv_KT_model} and Proposition
\ref{prop:equiv_sympl} is
\[
\begin{array}{lcccl}
 & \deg & G\text{-module} & \Gm\text{-weight} & d \\[2pt] \hline
\rule{0pt}{2.4ex}
x_i & 0 & V^\vee & 0 & \\
y_i & -1 & V \otimes \chi^{-1} & 0 & \partial s / \partial x_i \\
z & -1 & \chi^{-1} & 0 & s \\
w_\xi & -2 & \mathfrak{g} \otimes \chi^{-1} & 0 &
\textstyle\sum_i a_\xi(x_i)\, y_i + \langle d\chi, \xi \rangle\, z \\[2pt] \hline
\rule{0pt}{2.4ex}
t & 0 & \chi & 1 & \\
\widetilde{y}_i = t y_i & -1 & V & 1 &
\partial \widetilde{f} / \partial x_i \\
\zeta = z & -1 & \chi^{-1} & 0 &
\partial \widetilde{f} / \partial t \\
\widetilde{w}_\xi = t w_\xi & -2 & \mathfrak{g} & 1 &
\textstyle\sum_i a_\xi(x_i)\, \widetilde{y}_i
+ \langle d\chi, \xi \rangle\, t \zeta
\end{array}
\]
The upper block generates $A_G$; the two blocks together generate
$\widetilde{A}_G$, after inverting $t$. Multiplication by $t$ carries the rows
$y_i$ and $w_\xi$ of the upper block to $\widetilde{y}_i$ and
$\widetilde{w}_\xi$, raising the $\Gm$-weight by $1$ and untwisting the
character, and the weight-$0$ generators are exactly those of the upper block,
which is the statement $A_G = \widetilde{A}_G^{\Gm}$ read on generators. The
generator $\zeta$ is the one degree $-1$ generator of the symplectic model that
is not of the form $t \cdot (\,\cdot\,)$, and it is the contact direction.
\end{remark}

The generators $w_\xi$ cannot be omitted: the next proposition computes the
cohomological invariant that forces them, and checks that the model has it.

\begin{proposition}[The degree $-2$ obstruction]\label{prop:degree_two}
Let $Z$ be a $-1$-shifted contact derived Artin stack and $z \in Z(\K)$ a point
whose stabilizer acts on $\cL_z$ through $\chi$.
\begin{enumerate}
\item We have $h^{-2}(\LL_Z|_z) \simeq \mathrm{Lie}\,\ker(\chi)$. Consequently neither
an \'etale morphism \[([\Delta\mathrm{loc}(s)/G], 0) \to (Z, z)\] inducing an
isomorphism of stabilizers nor a smooth atlas of $Z$ by contact Darboux schemes
$\Delta\mathrm{loc}(s)$ exists near a point with $\dim \ker(\chi) > 0$. Both
exist when $\dim G = 0$.
\item Suppose $s(0) = 0$ and $ds(0) = 0$, so that the origin lies in the
classical truncation. Then the model has the correct invariant:
$h^{-2}(\LL_{[\Delta\mathrm{loc}_G(s)/G]}|_0) \simeq \ker(d\chi)$, and the
presentation of Definition \ref{def:equiv_discriminant} is minimal at the origin
in degree $-2$ if and only if $d\chi = 0$. Otherwise, exactly one generator
$w_\xi$ cancels against $z$.
\end{enumerate}
\end{proposition}

\begin{proof}
(1) Let $\widetilde{z} \in \widetilde{Z}(\K)$ be a lift of $z$. The stabilizer
of $\widetilde{z}$ is $\ker(\chi)$, since the stabilizer of $z$ acts on the
fiber $p^{-1}(z)$ through $\chi$. 
Write $\mathfrak{k} := \ker(d\chi) =
\mathrm{Lie}\,\ker(\chi)$, so that
\[h^1(\LL_{\widetilde{Z}}|_{\widetilde{z}}) \simeq \mathfrak{k}^\vee.\] 
As
$\widetilde{Z}$ is $-1$-shifted symplectic, $\LL_{\widetilde{Z}} \simeq
\TT_{\widetilde{Z}}[1]$ gives $h^i(\LL_{\widetilde{Z}}|_{\widetilde{z}}) \simeq
h^{-1-i}(\LL_{\widetilde{Z}}|_{\widetilde{z}})^\vee$, thus,
\[h^{-2}(\LL_{\widetilde{Z}}|_{\widetilde{z}}) \simeq \mathfrak{k}.\] The relative
cotangent sequence of the $\Gm$-torsor $p$ reads $p^* \LL_Z \to
\LL_{\widetilde{Z}} \to \cO$ with the third term in degree $0$, so
\[h^{-2}(\LL_Z|_z) \simeq \mathfrak{k}.\]
For the consequence, the cotangent
complex of $[\Delta\mathrm{loc}(s)/G]$ has amplitude $[-1,1]$, so its $h^{-2}$
vanishes, and a smooth morphism from a stack with vanishing $h^{-2}$ forces
\[h^{-2}(\LL_Z|_z) = 0\] on the image.

(2) At the origin the linear part of $d(w_\xi)$ is
$\langle d\chi, \xi \rangle z$, because the vector field $a_\xi$ vanishes at the
fixed point $0$. So, the linearized differential $\mathfrak{g} \to V \oplus \K$
in degrees $-2$ and $-1$ is $\xi \mapsto (0, \langle d\chi, \xi \rangle)$ and
$h^{-2} \simeq \ker(d\chi)$. Minimality at $0$ in degree $-2$ is the vanishing
of this linearized differential, which is the vanishing of $d\chi$. When
$d\chi \neq 0$, its image is the line spanned by $z$, so exactly one generator
cancels against $z$.
\end{proof}

With the model and its obstruction, it remains to show that every
$-1$-shifted contact structure on an affine quotient is of this form.

Two computations with $G = \Gm$ show the moment map first surviving and then
cancelling.

\begin{example}\label{ex:equiv_dloc}
Let $G = \Gm$, distinct from the structural $\Gm$, act on $V = \mathbb{A}^2$
with coordinates $u, v$ of weights $1$ and $-1$, let $\chi$ be trivial, and let $s = uv$, which is invariant and has
$s(0) = 0$, $ds(0) = 0$. Write $\xi$ for the generator of $\mathfrak{g} = \K$, so
that $a_\xi = u\, \partial/\partial u - v\, \partial/\partial v$ and $\langle
d\chi, \xi \rangle = 0$. Then $A_G$ has generators $u, v$ in degree $0$, the
generators $y_u, y_v, z$ in degree $-1$ with
\[ d(y_u) = v, \qquad d(y_v) = u, \qquad d(z) = uv, \]
and one generator $w$ in degree $-2$ with
\[ d(w) = u\, y_u - v\, y_v, \qquad
d^2(w) = uv - vu = 0 . \]
Here $\ker(\chi) = G$, so $\mathfrak{k} = \K$ and $d\chi = 0$: the presentation
is minimal in degree $-2$, $h^{-2} = \K$, and $w$ cannot be cancelled. This is the smallest
case in which $[\Delta\mathrm{loc}(s)/G]$ is not a local model, the classical
truncation of $[\Delta\mathrm{loc}_G(s)/G]$ being $B\Gm$. The symplectification
is $\dCrit([\,t\,uv\,/\,G\,])$ on $[(\mathbb{A}^2 \times \Gm)/G]$, where the
factor $\Gm$ is the structural one and $t$ has $G$-weight $0$.

If instead $G = \Gm$ acts on $V = \mathbb{A}^1$ with weight $-1$ and
$\chi(g) = g^2$,
then $s = x^2$ is $\chi^{-1}$-semi-invariant, $a_\xi = -x\, \partial/\partial x$
and $\langle d\chi, \xi \rangle = 2$, so $d(w) = -x y + 2z$. Now $d\chi \neq 0$,
$\ker(\chi) = \mu_2$ is finite, and $w$ cancels against $z$.
\end{example}

\subsection{The Affine Equivariant Contact Darboux Theorem}
\label{subsec:affine}

Everything is now in place for the central statement: on an affine slice, a
$-1$-shifted contact structure on the quotient is the standard one on an
equivariant discriminant locus. The global theorem of \S\ref{subsec:global}
reduces to it.

\begin{proposition}[Equivariant contact Darboux theorem for affine slices]
\label{prop:equiv_darboux}
Let $G$ be a linearly reductive group over $\K$ acting on an affine derived
$\K$-scheme $S$ of finite presentation, fixing a classical point $s_0 \in
S(\K)$. Suppose the quotient stack $[S/G]$ carries a $-1$-shifted contact
structure $(\cL, \alpha)$, and let $\chi$ be the character of the $G$-action on
$\cL_{s_0}$. Then there exist 
\begin{enumerate}[(i)]
\item a $G$-invariant affine open $S' \subseteq S$
containing $s_0$, 
\item a finite-dimensional $G$-representation $V$, and a
$\chi^{-1}$-semi-invariant function $s \in \cO(V)$ with $s(0) = 0$ and
$ds(0) = 0$, together with 
\item a $G$-equivariant equivalence
$S' \simeq \Delta\mathrm{loc}_G(s)$ carrying $s_0$ to $0$ and the contact
structure to the standard one of Definition \ref{def:equiv_discriminant},
\end{enumerate}

such that the presentation $A_G$ is generated over $\cO(V) = \mathrm{Sym}(V^\vee)$ by the
$G$-modules $V \otimes \chi^{-1}$ and $\chi^{-1}$ in degree $-1$, spanned by the
$y_i$ and by $z$, and $\mathfrak{g} \otimes \chi^{-1}$ in degree $-2$, spanned by
the $w_\xi$.
\end{proposition}

\begin{proof}
There are five steps. Steps 1 and 2 make the construction of
\cite[Theorem 3.7]{kib2} equivariant: the contact line is trivialized by a
semi-invariant section, and every choice in the minimal-standard-form induction
is averaged. Step 3 records what this produces, a chart whose degree $-2$
generators the construction leaves unconstrained. Steps 4 and 5 pin down the
weights and, on the symplectification, the differentials of those generators,
which is where the contact moment map appears.
\medskip

\bfem{Step 1 (Trivialization).} By the semi-invariant trivialization of
\S\ref{subsec:slices}, after shrinking $S$ we may trivialize $\cL$ by a
$\chi$-semi-invariant section, replacing
$\alpha$ by an ordinary $1$-form $\alpha_0$ of degree $-1$ satisfying
$g^* \alpha_0 = \chi(g)^{-1} \alpha_0$.
\medskip

\bfem{Step 2 (Equivariantizing the induction).}
The construction of a minimal standard form presentation and of the Darboux
coordinates for $\alpha_0$ in the proof of \cite[Theorem 3.7]{kib2}
(following \cite[Section 4]{Brav}) is an induction on degree whose steps are of
six kinds. Choices of generators of homotopy groups and of finite-dimensional
generating subspaces of obstruction modules, lifts along surjections of modules,
and chain homotopies contracting exact forms, in particular the passage from
closed to exact of \cite[Proposition 5.7]{Brav}, are made equivariant by Lemma
\ref{lem:reynolds}(1),(2) and Proposition \ref{prop:equiv_standard_form}.

Choices of a basis dual to a non-degenerate pairing are available equivariantly
because rational $G$-modules are semisimple in characteristic zero. 

Zariski
localization at the marked point, used twice to invert bundle maps that are
isomorphisms there, stays equivariant because the isomorphism locus $S_0$ is
$G$-invariant and contains $s_0$: its complement is closed and $G$-invariant, so
its ideal $I \subseteq H^0(\cO(S))$ is a $G$-stable submodule containing some
$f_0$ with $f_0(s_0) \neq 0$, and since evaluation at the fixed point $s_0$
commutes with the Reynolds operator, $f := R(f_0)$ is invariant with
$f(s_0) = f_0(s_0) \neq 0$ and $s_0 \in D(f) \subseteq S_0$. 

The rescaling of
the contact generator by the unit $u := \iota_{\partial/\partial z}\alpha_0$
also stays equivariant: $u$ is $\psi$-semi-invariant for some character $\psi$,
evaluating at $s_0$ gives $u(s_0) = \psi(g)u(s_0)$ with $u(s_0) \neq 0$, so
$\psi = 1$, $u$ is invariant, and the rescaled generator keeps its weight. 

The degree-$0$ generators are chosen as a $G$-stable
lift $V^\vee$ of the cotangent space of $S$ at $s_0$.
\smallskip

\bfem{Step 3 (The chart).} This produces a $G$-invariant affine open $S' \ni s_0$
and a $G$-equivariant
equivalence of $S'$ with the spectrum of a standard-form cdga $A$ on which $G$
acts linearly on the generators, with Darboux coordinates in which $\alpha_0 =
d_{\DR}z + \sum_i y_i\, d_{\DR}x_i$, minimal at $s_0$ in degrees $0$ and $-1$.
Minimality in degree $-2$ is not asserted, and by Proposition
\ref{prop:degree_two}(2), applied to the model with which Steps 4 and 5 identify
the chart, it fails whenever $d\chi \neq 0$. Besides $x_i, y_i, z$, the
cdga $A$ carries $m = \dim G$ generators in degree $-2$, the relative dimension
of the atlas $S' \to [S'/G]$. The normal form of \cite[Theorem 3.7]{kib2}
constrains neither these generators nor their differentials.
\smallskip

\bfem{Step 4 (Weights in degrees $0$ and $-1$).} Semi-invariance of $\alpha_0$
forces the weights of the generators in degrees $0$ and $-1$. Writing the degree-$0$ generators as a basis of $V^\vee$, the relation
$g^* \alpha_0 = \chi(g)^{-1} \alpha_0$ holds if and only if the span of the $y_i$
is the $G$-module $V \otimes \chi^{-1}$ and $z$ spans $\chi^{-1}$. The structural
equation $d(z) = s$ then forces $s(g \cdot x) = \chi(g)^{-1} s(x)$, and
$d(y_i) = \partial s/\partial x_i$ is equivariant for these weights.
Since $s_0$ lies in the classical truncation
$\{s = 0\} \cap \mathrm{Crit}(s)$, we have $s(0) = 0$ and $ds(0) = 0$.
\medskip

\bfem{Step 5 (The generators in degree $-2$).} It remains to identify these, which
we do on the symplectification. By the semi-invariant trivialization, the
principal $\Gm$-bundle associated to $\cL|_{S'}$ is $\Spec(\widetilde{A})$ with
$\widetilde{A} = A[t,t^{-1}]$, so the symplectification of $[S'/G]$ is $[\Spec(\widetilde{A})/G]$
with Liouville form $\lambda = t\alpha_0$ and $\omega = d_{\DR}\lambda$. Setting
$\widetilde{y}_i := t y_i$ and $\zeta := z$, Proposition \ref{prop:equiv_sympl}(1)
identifies $x_i, t$ in degree $0$ and $\widetilde{y}_i, \zeta$ in degree $-1$
with the Koszul--Tate generators of $\dCrit(\widetilde{f})$ for the
$G$-invariant potential $\widetilde{f}(x,t) = t \cdot s(x)$ on $V \times \Gm$.
The degree $-2$ generators of $A$ have $\Gm$-weight $0$, so multiplied by $t$
they are those of $\widetilde{A}$. By the standard form of the derived critical
locus of an invariant function on a quotient stack \cite[Proposition 3.13]{KPS},
\cite[Theorem 2.10]{BenBassat2015} they are indexed by $\mathfrak{g}$, in agreement with
the count $m = \dim G$, and $\widetilde{w}_\xi$ has differential the shifted
moment map
\[ d(\widetilde{w}_\xi)
= \sum_{i=1}^n a_\xi(x_i)\, \widetilde{y}_i
+ \langle d\chi, \xi \rangle\, t \zeta , \]
the image of $a_\xi$ under the identification of the degree $-1$ generators with
$T_{V \times \Gm}$, where $a_\xi(t) = \langle d\chi, \xi \rangle t$ because $G$
acts on $t$ through $\chi$. Setting $w_\xi := t^{-1}\widetilde{w}_\xi$ gives
\[ d(w_\xi) = \sum_{i=1}^n a_\xi(x_i)\, y_i + \langle d\chi, \xi \rangle\, z , \]
so $A$ is the algebra $A_G$ of Definition \ref{def:equiv_discriminant} and
$\Spec(A) = \Delta\mathrm{loc}_G(s)$, with the weights of Lemma
\ref{lem:equiv_KT_model}(1). For $\dim G = 0$ there are no such generators and
$A$ is the Koszul--Tate algebra of Lemma \ref{lem:KT_model}.
\end{proof}

Proposition \ref{prop:equiv_darboux} is the contact counterpart of the
equivariant Darboux theorems for \cite{Brav} established in the symplectic
setting by Descombes \cite{Descombes} and Park \cite{ParkPush}, and its proof is
the same equivariantization procedure applied to the construction of
\cite[Theorem 3.7]{kib2}, together with the identification of the degree $-2$
part of the chart, which that construction leaves free and which the geometry of
the quotient fixes to be the contact moment map. Its graded, weight-one aspect is
the one arising in dimensional reduction \cite{Kinjo}.

\subsection{The Global Statement}
\label{subsec:global}

\begin{theorem}[Equivariant Contact Darboux Theorem] \label{thm:contact_darboux}
Let $Z$ be a quasi-separated, locally finitely presented $-1$-shifted contact
derived Artin stack over $\K$ with affine stabilizers.
\begin{enumerate}
\item $Z$ admits a smooth atlas by contact Darboux charts (Definition
\ref{def:equiv_discriminant}), each carrying as many generators in degree $-2$
as its relative dimension. After a trivialization of the
contact line, the pullback of the contact form along a chart is
equivalent to the standard form $\alpha$ except in degree $-2$. If $Z$ admits an
\'etale atlas, for instance if $Z$ is a derived scheme or a derived
Deligne--Mumford stack, the charts are the contact Darboux schemes
$\Delta\mathrm{loc}(s)$ of Lemma \ref{lem:KT_model} and the equivalence of
contact structures holds without exception.
\item Let $z \in Z(\K)$ be a point whose stabilizer $G = \mathrm{Aut}(z)$ is
linearly reductive, and let $\chi \colon G \to \Gm$ be the character through which
$G$ acts on the fiber $\cL_z$ of the contact line bundle
(\S\ref{subsec:contact}). Then there exist a finite-dimensional
$G$-representation $V$, a $\chi^{-1}$-semi-invariant regular function
$s \in \cO(V)$ with $s(0) = 0$ and $ds(0) = 0$, and a pointed \'etale morphism
of $-1$-shifted contact stacks
\[ ([\Delta\mathrm{loc}_G(s)/G],\, 0) \longrightarrow (Z, z) , \]
carrying the standard contact structure of Proposition \ref{prop:equiv_sympl}(2) to
the given one and inducing an isomorphism of stabilizers at $0$.
\end{enumerate}
\end{theorem}

\begin{proof}
Part (1) is the atlas of \cite[Theorem 3.7]{kib2}, whose charts are
contact Darboux charts in the sense of Definition \ref{def:equiv_discriminant}. When the atlas can be taken \'etale, Proposition \ref{prop:equiv_darboux} at
trivial $G$ identifies each chart with $\Delta\mathrm{loc}(s)$. Part (2) reduces to the affine theorem in four steps: Step 1 produces an
affine slice through $z$, Step 2 applies Proposition \ref{prop:equiv_darboux}
on it, and Steps 3 and 4 verify that the resulting \'etale morphism matches the
contact structures, by comparing the two symplectifications.
\begin{enumerate}
\item By the contact Darboux theorem for derived Artin stacks \cite[Theorem
3.7]{kib2}, $Z$ admits a smooth atlas by affine charts in standard
contact form, each carrying as many generators in degree $-2$ as its relative
dimension, the comparison of the trivialized contact form with the standard one,
exact except in degree $-2$, being part of the statement of \cite[Theorem
3.7]{kib2}. If the atlas may be taken \'etale there are none, and
Proposition \ref{prop:equiv_darboux} with trivial $G$ identifies each chart with
$\Delta\mathrm{loc}(s)$ together with its contact structure.

\item \bfem{Step 1 (Slice).} The classical truncation of $Z$ is quasi-separated,
locally of finite type over $\K$, with affine stabilizers, and $G =
\mathrm{Aut}(z)$ is linearly reductive. By Theorem \ref{thm:AHR} there is an affine scheme with
$G$-action fixing a point $s_0$ and a pointed \'etale morphism from its
quotient stack to the truncation of $Z$, inducing an isomorphism of stabilizers
at $s_0$. Since the \'etale site is invariant under derived thickenings, this
morphism lifts uniquely to a pointed \'etale morphism $u \colon [S/G] \to Z$
with $S$ affine derived and $G$ acting on $S$ fixing $s_0$ (see
\cite[Proposition 3.13]{KPS} for the derived enhancement). The contact structure pulls back along the \'etale
$u$ to a $-1$-shifted contact structure on $[S/G]$, and $G$ acts on
$\cL_{s_0}$ through the character $\chi$, by the functoriality recorded in
\S\ref{subsec:contact}.
\medskip

\noindent \bfem{Step 2 (Equivariant Darboux on the slice).} By Proposition
\ref{prop:equiv_darboux}, after shrinking $S$ to a $G$-invariant affine open
neighborhood of $s_0$ we may assume \[S \simeq \Delta\mathrm{loc}_G(s) =
\Spec(A_G)\] $G$-equivariantly, with $A_G$, $V$, $s$, and $\alpha = d_{\DR}z +
\sum_i y_i\, d_{\DR}x_i$ as in the proposition.
\medskip

\noindent \bfem{Step 3 (Symplectification of the chart).} Symplectification commutes with
the \'etale morphism $u$ (\S\ref{subsec:contact}), so it suffices to symplectify
$[\Spec(A_G)/G]$. The contact line bundle over $\Spec(A_G)$ is equivariantly
trivialized by construction, so the associated principal $\Gm$-bundle is
$\Spec(\widetilde{A}_G)$ with $\widetilde{A}_G := A_G[t, t^{-1}]$, where $\Gm$
scales $t$ with weight $1$ and $G$ acts on $t$ through $\chi$. The tautological
Liouville form is $\lambda = t\alpha$ and the symplectic form is
$\omega = d_{\DR}(t\alpha)$.

Set 
\[\widetilde{y}_i := t y_i,\quad \zeta := z, \quad \text{and}\quad \widetilde{w}_\xi := t w_\xi.\]
By Proposition \ref{prop:equiv_sympl}, $\widetilde{A}_G$ is the standard-form
Koszul--Tate cdga representing the derived
critical locus $\dCrit([\,\widetilde{f}\,/\,G\,])$ of the potential $\widetilde{f}(x,t) = t \cdot s(x)$ on the
quotient stack $[(V \times \Gm)/G]$, with Koszul--Tate generators $x_i, t$ in
degree $0$, $\widetilde{y}_i, \zeta$ in degree $-1$ and $\widetilde{w}_\xi$ in
degree $-2$. The potential is homogeneous of $\Gm$-weight $1$ and $G$-invariant
since $t$ and $s$ carry the characters $\chi$ and $\chi^{-1}$, its standard
Darboux form is $\omega = d_{\DR}(t\alpha)$, and the weight $0$ part of $t\alpha$
is the standard Darboux contact $1$-form $\alpha = d_{\DR}z + \sum_{i=1}^n y_i\,
d_{\DR}x_i$ in local coordinates \cite{kib1, kib2}.

Note that $A_G = \widetilde{A}_G^{\Gm}$, so $\Spec(A_G)$ is the $\Gm$-reduction
of the symplectified chart.
\medskip

\noindent \bfem{Step 4 (Conclusion).} We obtain a commutative square, pointed at $z$ and
inducing an isomorphism of stabilizers, with \'etale horizontal arrows and
Cartesian by Step 3:
\[
\begin{tikzcd}[scale=1.5, row sep=large, column sep=large]
{\dCrit([\,t \cdot s\,/\,G\,])} \arrow[r, "\widetilde{u}"]
\arrow[d, "\pi_{\Gm}"'] &
\widetilde{Z} \arrow[d, "p"] \\
{[\Delta\mathrm{loc}_G(s) / G]} \arrow[r, "u"] & Z
\end{tikzcd}
\]
The morphism $u$ is one of contact stacks: the contact structure on $[S/G]$ is
the pullback of the given one along $u$ by Step 1, and Step 2 identifies it with
the standard structure of $[\Delta\mathrm{loc}_G(s)/G]$. As $u$ is pointed
\'etale, carries the standard contact structure to the given one, and induces an
isomorphism of stabilizers at $0$ by Step 1, this proves (2).
\end{enumerate}
\end{proof}

In particular, $Z$ is smooth-locally around $z$ equivalent to the quotient stack
$[\Delta\mathrm{loc}_G(s)/G]$. The same statement for the naive quotient
$[\Delta\mathrm{loc}(s)/G]$ holds only if $\dim \ker(\chi) = 0$, by Proposition
\ref{prop:degree_two}.
\vspace{2em}

\begin{corollary}[Schematic contact Darboux]\label{cor:schematic}
Let $X$ be a locally finitely presented $-1$-shifted contact derived
$\K$-scheme and $x \in X(\K)$. Then there exists 
\begin{enumerate}[(i)]
\item a Zariski open $X' \subseteq X$
containing $x$, 
\item a smooth affine $\K$-scheme $V$, and a regular function
$s \in \cO(V)$ with $s(0) = 0$ and $ds(0) = 0$, together with 
\item an equivalence of
$-1$-shifted contact derived schemes $X' \simeq \Delta\mathrm{loc}(s)$ carrying
$x$ to $0$.
\end{enumerate}
Thus, every $-1$-shifted contact derived scheme is Zariski-locally the
derived discriminant locus of a regular function. Via Lemma \ref{lem:KT_model}
this is the schematic Darboux theorem of \cite[Section 3]{kib1} in the form
used here.
\end{corollary}

\begin{proof}
Choose an affine open $S \subseteq X$ of finite presentation containing $x$ and
apply Proposition \ref{prop:equiv_darboux} with $G$ trivial: it returns an
affine open $S' \subseteq S$ containing $x$ and an equivalence $S' \simeq
\Delta\mathrm{loc}(s)$ carrying $x$ to $0$ and the contact structure to the
standard one, with $s(0) = 0$ and $ds(0) = 0$ since $x$ lies in the classical
truncation $\{s = 0\} \cap \mathrm{Crit}(s)$.
\end{proof}

A $\chi^{-1}$-semi-invariant function on $V$ is the same datum as a $G$-invariant
section of the equivariant line bundle $\cO_V \otimes \chi^{-1}$, equivalently a
section of a line bundle on $[V/G]$. This reconciles the phrasing ``relative
section'' of the abstract with the phrasing ``regular function'' used in Theorem
\ref{thm:A} and here: the potential $s$ is a function on the atlas and a
relative section on the quotient. In the same way $\Delta\mathrm{loc}_G(s)$ is the derived zero locus
of the $1$-jet of $s$ taken relative to $[V/G]$, the generators $w_\xi$ resolving
the relations imposed by the $\mathfrak{g}$-action.

\subsection{Quivers with Potential}
\label{subsec:quivers}

Quivers with potential are the case in which the degree $-2$ generators are
unavoidable at every point. We keep the notation of
\S\ref{subsec:quiverprelim}: $Q$ is a finite quiver, $W$ a potential, $d$ a
dimension vector, $V := \mathrm{Rep}(Q,d)$, $G := G_d$, and
$W_d = \mathrm{Tr}(W)$. In the Koszul--Tate algebra of Definition
\ref{def:equiv_discriminant} attached to this datum we write $y_a$ for the
degree $-1$ generator dual to $x_a$ under the trace pairing, so that
$y_a \in \Hom(\K^{d_{h(a)}}, \K^{d_{t(a)}})$ and $d(y_a) = \partial_a W$.

\begin{proposition}[Contact models from quivers with potential]\label{prop:quiver}
Let $\chi$ be the trivial character of $G_d$.
\begin{enumerate}
\item $(G_d, \chi, V, W_d)$ is a datum as in Definition
\ref{def:equiv_discriminant}, and $[\Delta\mathrm{loc}_{G_d}(W_d)/G_d]$ is a
$-1$-shifted contact derived Artin stack with trivial contact line.
\item The contact moment map is the moment map of the doubled quiver: for
$\xi = (\xi_i) \in \mathfrak{g} = \bigoplus_i \mathfrak{gl}_{d_i}$,
\[ \langle \mu, \xi \rangle = \sum_{i \in Q_0} \mathrm{Tr}(\xi_i\, \mu_i),
\qquad
\mu_i = \sum_{h(a) = i} x_a y_a \; - \sum_{t(a) = i} y_a x_a . \]
\end{enumerate}
\end{proposition}

\begin{proof}
(1) The trace of a cyclic word is invariant under simultaneous conjugation, so
$W_d$ is $G_d$-invariant, which for trivial $\chi$ is the semi-invariance
required in Definition \ref{def:equiv_discriminant}. Proposition
\ref{prop:equiv_sympl}(2) then supplies the contact structure, with contact
line $\cO \otimes \chi = \cO$.

(2) With the convention of Definition \ref{def:equiv_discriminant},
\[ a_\xi(x_a) = \frac{d}{d\epsilon}\Big|_{\epsilon = 0}
\exp(\epsilon \xi_{h(a)})\, x_a \exp(-\epsilon \xi_{t(a)})
= \xi_{h(a)} x_a - x_a \xi_{t(a)} , \]
and $\langle d\chi, \xi \rangle = 0$, so
$\langle \mu, \xi \rangle = \sum_a \mathrm{Tr}\bigl((\xi_{h(a)} x_a - x_a
\xi_{t(a)})\, y_a\bigr)$. Cyclicity of the trace turns the second summand into
$\mathrm{Tr}(\xi_{t(a)}\, y_a x_a)$, and collecting the arrows by head and by
tail gives the stated formula, which is \eqref{eq:preproj} with $y_{a^*}$
replaced by $y_a$. The two differ in where the second variable lives: in
\eqref{eq:preproj} it is a coordinate on $T^*\mathrm{Rep}(Q,d)$, here it is a
Koszul--Tate generator of degree $-1$ with $d(y_a) = \partial_a W$. As a check,
$d\langle \mu, \xi \rangle = a_\xi(W_d) = 0$ by invariance, as Lemma
\ref{lem:equiv_KT_model}(1) requires.
\end{proof}

Three consequences follow from the proposition together with Proposition
\ref{prop:degree_two} and Euler's relation.

\begin{corollary}[Obstruction and truncation]\label{cor:quiver}
In the context of Proposition \ref{prop:quiver}:
\begin{enumerate}
\item At a point $x$ with stabilizer $G_x \subseteq G_d$,
\[ h^{-2}\bigl(\LL_{[\Delta\mathrm{loc}_{G_d}(W_d)/G_d]}\big|_x\bigr)
\simeq \mathrm{Lie}\, G_x . \]
The scalars act trivially on $V$, so $\dim G_x \geq 1$ for every $x$, and
$[\Delta\mathrm{loc}(W_d)/G_d]$ is a local model at no point of the stack.
\item If $W$ is homogeneous of degree $m$, then $W_d$ vanishes on
$\mathrm{Crit}(W_d)$ and
\[\bigl(\Delta\mathrm{loc}_{G_d}(W_d)\bigr)^{\mathrm{cl}} =
\mathrm{Rep}(\mathrm{Jac}(Q,W), d).\]
\item If every cycle of $W$ has length at least $2$, then $W_d(0) = 0$ and
$dW_d(0) = 0$, so the origin lies in the truncation and the presentation of
Definition \ref{def:equiv_discriminant} is minimal there in degree $-2$ in the
sense of Proposition \ref{prop:degree_two}(2).
\end{enumerate}
\end{corollary}

\begin{proof}
(1) The stabilizer $G_x$ acts on the contact line through $\chi$, which is
trivial, so $\ker(\chi|_{G_x}) = G_x$ and Proposition \ref{prop:degree_two} gives the
displayed isomorphism. The diagonal $\Gm \subseteq G_d$ acts trivially on $V$,
being conjugation by scalars, so it lies in every $G_x$. Hence, $h^{-2} \neq 0$
everywhere, while the cotangent complex of $[\Delta\mathrm{loc}(W_d)/G_d]$ has
amplitude $[-1,1]$.

(2) Euler's relation for the homogeneous function $W_d$ of degree $m$ in the
matrix entries reads $\sum_{a} \mathrm{Tr}(x_a\, \partial_a W) = m\, W_d$, and
each $\partial_a W$ vanishes on $\mathrm{Crit}(W_d)$, so $W_d$ does too as
$\mathrm{char}\,\K = 0$. By Lemma \ref{lem:equiv_KT_model}(2) the classical
truncation is $\{W_d = 0\} \cap \mathrm{Crit}(W_d) = \mathrm{Crit}(W_d)$, which
is $\mathrm{Rep}(\mathrm{Jac}(Q,W), d)$ by \eqref{eq:crit_jac}.

(3) A cycle of length $\ell$ contributes to $W_d$ a polynomial homogeneous of
degree $\ell$ in the matrix entries, so $\ell \geq 2$ for every cycle gives
$W_d(0) = 0$ and $dW_d(0) = 0$. The origin is $G_d$-fixed and $d\chi = 0$, so
Proposition \ref{prop:degree_two}(2) applies.
\end{proof}

The smallest quiver already exhibits the whole structure.

\begin{example}[One loop]\label{ex:jordan}
Let $Q$ be the Jordan quiver, one vertex with one loop $x$,
\[
\begin{tikzcd}[column sep=large]
\bullet \arrow[loop right, "x"]
\end{tikzcd}
\]
Let $W = x^{\ell+1}$ with $\ell \geq 1$, and let $d \geq 1$, so that
$V = \mathfrak{gl}_d$, $G = GL_d$ acts by conjugation and
$W_d = \mathrm{Tr}(x^{\ell+1})$. Then $A_G$ is generated over
$\cO(\mathfrak{gl}_d)$ by $y$ and $z$ in degree $-1$ and $w_\xi$, $\xi \in
\mathfrak{gl}_d$, in degree $-2$, with
\[ d(y) = (\ell+1)\, x^\ell, \qquad d(z) = \mathrm{Tr}(x^{\ell+1}), \qquad
d(w_\xi) = \mathrm{Tr}\bigl(\xi\, [x, y]\bigr) , \]
the moment map being the commutator $[x,y]$, so that \eqref{eq:preproj} for the
doubled Jordan quiver is recovered, and $d^2(w_\xi) = \mathrm{Tr}(\xi [x,
(\ell+1)x^\ell]) = 0$. Here, $\mathrm{Jac}(Q,W) = \K[x]/(x^\ell)$, so by
Corollary \ref{cor:quiver}(2) the classical truncation is the stack
$[\{x \in \mathfrak{gl}_d : x^\ell = 0\} / GL_d]$ of endomorphisms annihilated by
$T^\ell$, and $h^{-2} \simeq \mathrm{Lie}\, G_x \neq 0$ at every point of it. For
$\ell = 1$ this is $BGL_d$.
\end{example}

\begin{remark}[Homogeneous potentials split off an odd line]\label{rem:quiver_split}
The generator $z$ does no cutting when $W$ is homogeneous. With $m$ as in
Corollary \ref{cor:quiver}(2), put $\eta := \frac{1}{m} \sum_a
\mathrm{Tr}(x_a y_a)$, of degree $-1$. Then, $d(\eta) = \frac{1}{m} \sum_a
\mathrm{Tr}(x_a\, \partial_a W) = W_d$ by Euler's relation, so $z' := z - \eta$
is a cycle and
\[ A_{G_d} \simeq B \otimes_\K \K[z'], \qquad d(z') = 0, \]
where $B$ is the standard-form cdga of $\dCrit([\,W_d\,/\,G_d\,])$, the local
model of quiver Donaldson--Thomas theory. The underlying derived scheme is
therefore the symplectic one with a free odd generator adjoined, and what
distinguishes the contact model is the form $\alpha$, not the derived structure.
For a potential that is not homogeneous, no such $\eta$ need exist, $W_d$ need not
vanish on $\mathrm{Crit}(W_d)$, and the truncation
$\{W_d = 0\} \cap \mathrm{Crit}(W_d)$ can be a proper closed subscheme of
$\mathrm{Rep}(\mathrm{Jac}(Q,W), d)$. The smallest instance is the one-loop
quiver with $W = x^3 + x^2$ and $d = 1$: there $\partial W = x(3x + 2)$ has
critical locus $\{0, -2/3\}$ while $W_d(-2/3) = 4/27$, so the truncation is the
single point $\{0\}$ and the generator $z$ removes the other.
\end{remark}

\begin{remark}[Framing]\label{rem:quiver_framing}
Let $Q^{\mathrm{fr}}$ be obtained from $Q$ by adjoining a vertex $\infty$ with
$d_\infty = 1$ and finitely many arrows $\infty \to i$, and keep the same
potential, so that the framing arrows do not occur in $W$. Proposition
\ref{prop:quiver} applies verbatim to
$V^{\mathrm{fr}} = \mathrm{Rep}(Q^{\mathrm{fr}}, d)$ with the same group $G_d$,
and the framing data on which $G_d$ acts freely form a $G_d$-invariant open
subset. On it, $G_x$ is trivial, so Corollary \ref{cor:quiver}(1) gives
$h^{-2} = 0$: the linearized
differential $\mathfrak{g}_d \to V^{\mathrm{fr}}$ of the moment map is injective,
the generators $w_\xi$ cancel against $\dim G_d$ of the degree $-1$ generators,
and the model of Definition \ref{def:discriminant} is available again. The degree
$-2$ generators of Theorem \ref{thm:A} are therefore carried exactly by the
locus of positive-dimensional stabilizer, which is the complement of the free
locus that framed invariants are designed to avoid.
\end{remark}

\section{Conclusion and Outlook}
\label{sec:conclusion}

Theorem \ref{thm:contact_darboux} presents a $-1$-shifted contact derived Artin
stack near a point with linearly reductive stabilizer $G$ as the quotient
$[\Delta\mathrm{loc}_G(s)/G]$ of a contact Darboux chart. The chart is the
Koszul--Tate algebra of a derived discriminant locus, extended in degree $-2$ by
generators indexed by $\mathfrak{g}$ with differential the contact moment map.
Symplectifying it returns the local model of quotient Donaldson--Thomas theory,
and of the two group directions in the resulting square the $\Gm$ direction
produces the symplectic structure and the $G$ direction the degree $-2$
generators (Proposition \ref{prop:equiv_sympl}).

At a point $z \in Z(\K)$ whose
stabilizer acts on the contact line through $\chi$, Proposition
\ref{prop:degree_two} computes
\[ h^{-2}(\LL_Z|_z) \simeq \mathrm{Lie}\,\ker(\chi) , \]
so they vanish exactly when $\ker(\chi)$ is finite, and the quotient of the
discriminant locus alone is a local model at no point where they are nonzero.
For a derived scheme they vanish and Corollary \ref{cor:schematic} returns the
Zariski-local statement that a $-1$-shifted contact derived scheme is the
derived discriminant locus of a regular function.

Quivers with potential are the case in which $h^{-2}$ is nonzero at every point.
The contact moment map is the moment map of the doubled quiver, the scalars act
trivially on $\mathrm{Rep}(Q,d)$, and for a homogeneous potential the classical
truncation is $\mathrm{Rep}(\mathrm{Jac}(Q,W), d)$ (Proposition
\ref{prop:quiver}, Corollary \ref{cor:quiver}). A framing makes the action free
on an open locus and $h^{-2}$ vanishes there, so the degree $-2$ generators are
carried exactly by the locus of positive-dimensional stabilizer (Remark
\ref{rem:quiver_framing}).

Equivariant \'etale vanishing cycles are computed on a quotient by a reductive
group, which a smooth atlas alone does not provide, and Theorem
\ref{thm:contact_darboux}(2) supplies one at every point with linearly reductive
stabilizer. The companion paper \cite{Izbudak_Monodromic} uses part (1) alone.

The behavior of these models under a change of chart remains to be described.
A definition of an $n$-shifted contactomorphism $f \colon X_1 \to X_2$ is not yet
available. We expect it to consist of an equivalence of contact distributions
$\cK_1 \xrightarrow{\sim} f^* \cK_2$ compatible with the tangent complexes,
together with a condition on shifted $1$-forms, and the coordinate changes
between the Darboux models of Theorem \ref{thm:contact_darboux} are a concrete
place to look for the right condition. Such a definition is also a prerequisite for
any shifted analogue of Gray stability, and for a notion of equivalence of
Legendrian correspondences under which $\LFc(X)$ would be invariant.

\section*{Acknowledgements}
The author thanks Kadri \.{I}lker Berktav for supervision and for collaboration on the prior work this paper builds upon.

\bibliographystyle{amsalpha}
\bibliography{refss}

\end{document}